\documentclass[12pt, reqno]{amsart}
\usepackage{amsmath, amsthm, amscd, amsfonts, amssymb, graphicx, color}
\usepackage[bookmarksnumbered, colorlinks, plainpages]{hyperref}

\newtheorem{theorem}{Theorem}[section]
\newtheorem{lemma}[theorem]{Lemma}
\newtheorem{proposition}[theorem]{Proposition}
\newtheorem{corollary}[theorem]{Corollary}
\theoremstyle{definition}
\newtheorem{definition}[theorem]{Definition}
\newtheorem{example}[theorem]{Example}

\theoremstyle{remark}
\newtheorem{remark}[theorem]{Remark}
\numberwithin{equation}{section}

\begin{document}

\title[Baire 1 functions and the strong Choquet property]{Baire 1 functions and the strong Choquet property}

\author[\v Lubica Hol\'a]{\v Lubica Hol\'a}

\newcommand{\acr}{\newline\indent}

\address{\llap Slovak Academy of Sciences, Institute of Mathematics \acr
\v Stef\'anikova 49, 81473 Bratislava,
	\acr Slovakia}

\email{\textcolor[rgb]{0.00,0.00,0.84}{hola@mat.savba.sk}}



\bigskip

\bigskip

\bigskip

\bigskip

\begin{abstract}
Let $X$ be a Tychonoff space and $B_1(X)$ be the space of real-valued Baire 1 functions. Using the classical Lebesgue's characterization of Baire 1 functions and the strong Choquet game we prove that if $X$ is locally compact or a topological sum of $\sigma$-compact spaces, the space $B_1(X)$ equipped with the topology $\tau_{UC}$ of uniform convergence on compacta, is a strong Choquet space. For the Sorgenfrey line $\Bbb S$, the space $(B_1(\Bbb S),\tau_{UC})$ is of the first Baire category in itself.
\end{abstract}

\maketitle

\section{Introduction}

\bigskip

A topological space $X$ is Baire if the intersection of any sequence of open dense subsets of $X$ is dense in $X$.
Being a Baire space is an important topological property for a
space and it is therefore natural to ask when function spaces are Baire. An interesting problem for the space $(C(X),\tau_{UC})$,
of continuous real-valued functions equipped with the topology $\tau_{UC}$ of uniform convergence on compacta, is a characterization of a topological space
$X$, for which $(C(X),\tau_{UC})$ is Baire.
This problem was solved for a locally compact space $X$ by Gruenhage and Ma in \cite{GM}, see also \cite{Ma}.
In our paper we prove that if $X$ is locally compact or a topological sum of $\sigma$-compact spaces, the space $(B_1(X),\tau_{UC})$, of Baire 1 functions equipped with the topology of uniform convergence on compacta, is a strong Choquet space. The strong Choquet property is stronger than Baire property.

The problem for $(C(X),\tau_p)$, the space of continuous real-valued functions equipped with the topology $\tau_p$ of pointwise convergence, was solved independently by
Pytkeev \cite{Py}, Tkachuk \cite{Tk} and van Douwen \cite{Do} and for the space $(B_1(X),\tau_p)$, of Baire 1 functions equipped with the topology of pointwise convergence, by Osipov \cite{Os}. Osipov in \cite{Os} also characterized a topological space $X$, for which $(B_1(X),\tau_p)$ is strong Choquet.

\section{Preliminaries}

\bigskip

Throughout this paper, all spaces are assumed to be Tychonoff.
We denote $\omega$ the set of
non-negative integers and $\Bbb{R}$ the space of real numbers with the usual metric.
The symbol $\overline A$ will stand for the closure of the set $A$ and Int$A$ for the interior of the set $A$.

Let $C(X,Y)$ be the space of all continuous functions from a topological space $X$ into a topological space $Y$. We use the following notation: $B_1(X,Y) = \{f \in Y^X: f$ is a pointwise limit of a sequence from $C(X,Y)\}$ and $F_\sigma(X,Y) = \{f \in Y^X: f^{-1}(V)$ is an $F_\sigma$ set for every open $V \subset Y\}$. The elements of $B_1(X,Y)$ are called functions
of the first Baire class or Baire 1 functions, and those of $F_\sigma(X,Y$) are called functions of the
first Borel class or $F_\sigma$ measurable functions. By $C(X)$ ($B_1(X)$, $F_\sigma(X)$) we denote the space of continuous (Baire 1, $F_
\sigma$ measurable) real-valued functions.

It is known that $B_1(X,Y) \subset F_\sigma(X,Y)$ for any topological space $X$ and any metric space $Y$ \cite{Ve}. If $Y = \{0, 1\}$ then $B_1(\Bbb R,Y) \ne F_\sigma(\Bbb R,Y)$. An overview of the results regarding the equality $B_1(X,Y) = F_\sigma(X,Y)$ can be found in the paper \cite{Ve}.
Baire in \cite{Ba} proved that if $X$ is an interval of reals $\Bbb R$, then $F_\sigma(X) = B_1(X)$. Lebesgue in \cite{Leb} proved that if $X$ is a metric space, then $F_\sigma(X) = B_1(X)$.

\bigskip

The following result of Laczkowich will be useful in our paper.

\begin{proposition}(\cite{La}) \label{la}
Let $X$ be a normal space. Then $F_\sigma(X) = B_1(X)$.
\end{proposition}

The following Theorem was proved by Lebesgue.

\begin{theorem} \label{HL15} \rm{(cf. \cite[p. 375]{Ku},\cite{BHH})}
	Let $X$ be a metric space. For $f\colon X \to \Bbb R$ the following conditions are equivalent:
	\begin{itemize}
		\item[(1)] A function $f$ is $F_\sigma$ measurable;
		\item[(2)] For each $\varepsilon >0$ there is a cover $(X_i)_{i \in \omega}$ of $X$ consisting of closed sets such that $\text{diam}f(X_i) \le\varepsilon$, for all $i \in \omega$.
	\end{itemize}
	\end{theorem}

\begin{remark} \label{lub}
It is easy to verify that Theorem \ref{HL15} works for any topological space $X$.
\end{remark}

We recall that a subset of $X$ that is the preimage of zero for a
certain function from $C(X)$ is called a zero-set. A subset $O \subset X$ is called
a cozero-set (or functionally open) of $X$ if $X \setminus O$ is a zero-set of $X$. The zero-sets are preserved by finite unions and countable
intersections. Hence cozero-sets are preserved by finite intersections and
countable unions. Countable unions of zero-
sets will be denoted by $Zer_\sigma$.

It is easy to check that $Zer_\sigma$-sets are preserved by countable unions and finite
intersections.
It is well
known that $f$ is of the first Baire class if and only if $f^{-1}(U)$ is $Zer_\sigma$-set for every
open $U \subset \Bbb R$ (see Exercise 3.A.1 in \cite{LMZ}).

\bigskip

We have the following analog of Lebesgue's theorem.

\begin{theorem} \label{HL16}
	Let $X$ be a Tychonoff space. For $f\colon X \to \Bbb R$ the following conditions are equivalent:
	\begin{itemize}
		\item[(1)] A function $f$ is Baire 1;
		\item[(2)] For each $\varepsilon >0$ there is a cover $(X_i)_{i \in \omega}$ of $X$ consisting of zero-sets such that $\text{diam}f(X_i) \le\varepsilon$, for all $i \in \omega$.
	\end{itemize}
	\end{theorem}
\begin{proof} (1) $\Rightarrow$ (2) Let $f$ be Baire 1.
For a fixed $\varepsilon >0$ we can find open intervals $B_0, B_1, ,B_2, \dots$ in $\Bbb R$ such that
$\Bbb R = \bigcup_{n\in\omega} B_n$ and the lenght of $B_n$, $|B_n| \le\varepsilon$ for all $n\in\omega$.
For $n\in\omega$, the set $f^{-1}(B_n)$ is $Zer_\sigma$-set, so it can be expressed as the union $\bigcup_{j\in\omega}F_{n,j}$ of zero-sets $F_{n,j}$.
Enumerate $\{F_{n,j}: n,j\in\omega\}$ as $\{X_i: i\in\omega\}$ and observe that the respective condition in (2) holds.

(2) $\Rightarrow$ (1) By (2), for each $k\in\omega$ we find zero-sets $X_{k,n}$, $n\in \omega$
such that $X=\bigcup_{n\in\omega}X_{k,n}$ and $\text{diam}f(X_{k,n})<1/(k+1)$ for all $n\in \omega$.
Let $U$ be an open set in $\Bbb R$. Define $\mathcal F=\{X_{k,n}: f(X_{k,n})\subset U\}$.
It is enough to show that $f^{-1}(U)=\bigcup\mathcal F$ since then $f^{-1}(U)$ is $Zer_\sigma$-set.

Firstly, if $x\in \bigcup\mathcal F$, then $x\in X_{k,n}$ for some $k,n\in \omega$ with $f(X_{k,n})\subset U$. Hence $f(x)\in U$ and so, $x\in f^{-1}(U)$.
Secondly, if $f(x)\in U$, we can find $k\in \omega$ such that the set $\{y\in \Bbb R: |f(x) - y|<1/(k+1)\}$ is contained
in $U$. Find $n\in \omega$ such that $x\in X_{k,n}$. Since $\text{diam}f(X_{k,n})<1/(k+1)$, we have $f(X_{k,n})\subset U$. Hence $x\in\mathcal F$ as desired.
\end{proof}

\bigskip

By $K(X)$ we denote the family of all nonempty compact subsets of $X$.
Denote by $\tau_{UC}$ the topology of uniform convergence on
compact sets on $\Bbb R^X$. This topology is induced by the
uniformity $\frak U_{UC}$ which has a base consisting of sets of the
form
$$W(K,\varepsilon )=\{(f ,g):\ \forall\ x\in K\ \ |f(x) - g(x)|<
\varepsilon \},$$
where $\noindent K\in K(X)$ and $\varepsilon >0$. The general
$\tau_{UC}$-basic neighborhood of $f\in \Bbb R^X$ will be denoted by
$W(f,K,\varepsilon )$, where
$$W(f,K,\varepsilon ) =\{g:\ \forall\ x\in K\ \ |f(x) - g(x)|<
\varepsilon \}.$$

In this paper, we are mainly interested in completeness properties. A topological space $X$ is \emph{Čech-complete} \cite{En}, if $X$ is a Tychonoff space and it is a $G_\delta$ set in one (equivalently, in all) of its compactifications. Each completely metrizable space is Čech-complete. A Tychonoff
topological space $X$ is called \emph{locally Čech-complete} \cite{En} if every point $x \in X$ has a
\v Cech-complete neighbourhood.

A set $A\subseteq X$ is of \emph{first Baire category}, if it is a countable union of nowhere dense sets; otherwise it is of \emph{second Baire category}. A space $X$ is called \emph{Baire} space, if every nonempty open subset of $X$ is of second Baire category. There is also a characterization of Baire spaces via topological games.

The Choquet game $G(X)$ of a topological space $X$ \cite{Kech} is a game between two players $\alpha$ and $\beta$. The player $\beta$ starts a play by selection a nonempty open subset $U_0$ of $X$. Then $\alpha $ chooses a nonempty open subset $V_0$ of $U_0$. In $n$th move, the player $\beta$ picks a nonempty open set $U_{n}
\subset V_{n-1}$, where $V_{n-1}$ is the previous move of $\alpha$-player, and $\alpha$ answers by selecting a nonempty open set $V_n \subset U_n$.
The player $\alpha$ wins the play $(U_i, V_i)_{i \in \omega}$, if $\bigcap_{n \in \omega} V_n \neq \emptyset$. Otherwise the player $\beta $ is said to have won the play. \\
We say that the player $\alpha$ has a winning strategy for the game $G(X)$ if there exists a strategy $s$, such that $\alpha $ wins all plays provided that he/she acts according to the strategy $s$. In this case, we say that $X$ is a \emph{Choquet space}.

The strong Choquet game $G^s(X)$ of a topological space $X$ \cite{Kech} is a game between two players $\alpha$ and $\beta$ similar to $G(X)$. The player $\beta$ starts a play by selection a pair $(x_0,U_0)$, where $U_0$ is an open subset of $X$ and $x_0 \in U_0$. Then $\alpha $ must play a nonempty open subset $V_0$ of $U_0$ with $x_0 \in U_0$.

In $n$th move, the player $\beta$ picks a pair $(x_n,U_n)$, where $U_n$ is an open subset of $V_{n-1}$, the previous move of $\alpha$-player, and $x_n \in U_n$. Then $\alpha$ answers by selecting an open set $V_n \subset U_n$ with $x_n \in V_n$.
The player $\alpha$ wins the play $(U_i, V_i)_{i \in \omega} $, if $\bigcap_{n \in \omega} V_n \neq \emptyset$. Otherwise the player $\beta $ is said to have won the play. \\
We say that the player $\alpha$ has a winning strategy for the game $G^s(X)$ if there exists a strategy $s$, such that $\alpha $ wins all plays provided that he/she acts according to the strategy $s$. In this case, we say that $X$ is a \emph{strong Choquet space}.

\bigskip

A topological space $X$ is hemicompact \cite{En} if in the family of all compact subspaces of $X$ ordered by inclusion there exists a countable cofinal subfamily. Every hemicompact space is $\sigma$-compact, but not vice versa. The space of rationals with the usual topology is a $\sigma$-compact space which is not hemicompact.

A topological space $X$ is a $q$-space, if for each point there exists a sequence $\{U_n:\ n\in \omega\}$ of neighbourhoods of that point so that if $x_n\in U_n$ for each $n$, then $\{x_n:\ n\in \omega \}$ has a cluster point \cite{MN}. The following result was proved in \cite{HH}.

\begin{corollary} \label{completeness}
Let $X$ be a Tychonoff space. The following are equivalent:
\begin{enumerate}

\item $(B_1(X)),\tau_{UC})$ is completely metrizable;
\item $(B_1(X),\tau_{UC})$ is \v Cech complete;
\item $(B_1(X),\tau_{UC})$ is a $q$-space;
\item $X$ is hemicompact.
\end{enumerate}
\end{corollary}

\bigskip

Now we show that every regular locally \v Cech complete space is a $q$-space. We will use an internal characterization of \v Cech-complete spaces.

\begin{theorem} (\cite{En}) \label{cech}
A Tychonoff space $X$ is \v Cech-complete if and only if there is a sequence $\{\mathcal G_n: n \in \omega\}$ of open covers of $X$ with the property that any family $\mathcal F$ of closed subsets of $X$, which has the finite intersection property and for every $n \in \omega$ there are $F \in \mathcal F$ and $G \in \mathcal G_n$ with $F \subset G$, has non-empty intersection.
\end{theorem}

\begin{proposition} \label{Q}
Let $X$ be a regular locally \v Cech complete space. Then $X$ is a $q$-space.
\end{proposition}
\begin{proof}
Let $x \in X$. There is a neighbourhood $V$ of $x$ which is a \v Cech-complete space with the induced topology from $X$. Without loss of generality we can suppose that $V$ is an open set in $X$. By Theorem \ref{cech} there is a sequence $\{\mathcal G_n: n \in \omega\}$ of open covers of $V$ with the property that any family $\mathcal F$ of closed subsets of $V$, which has the finite intersection property and for every $n \in \omega$ there are $F \in \mathcal F$ and $G \in \mathcal G_n$ such that $F \subset G$, has non-empty intersection. For every $n \in \omega$ let $G_n \in \mathcal G_n$ be such that $x \in G_n$. There is an open set $V_n$ in $X$ such that $x \in V_n \subset \overline{V_n} \subset G_n$. For every $n \in \omega$ put $U_n = \cap_{i \le n} V_i$. We claim that the sequence $\{U_n: n \in \omega\}$ is a sequence of open neighbourhoods of $x$ with the following property: if $x_n\in U_n$ for each $n$, then $\{x_n:\ n\in \omega \}$ has a cluster point. For every $m \in \omega$ put $F_m = \overline{\{x_k: k \ge m\}}$. Then $F_m$ is a closed subset of $G_m$. By the assumption the family $\{F_m: m \in \omega\}$ has a non-empty intersection. Thus the sequence $\{x_n: n \in \omega\}$ has a cluster point \cite{Ke}.
\end{proof}

\begin{corollary} \label{hemicompact}
Let $X$ be a Tychonoff space. The following are equivalent:
\begin{enumerate}
\item $(B_1(X)),\tau_{UC})$ is completely metrizable;
\item $(B_1(X),\tau_{UC})$ is \v Cech complete;
\item $(B_1(X),\tau_{UC})$ is locally \v Cech complete;
\item $(B_1(X),\tau_{UC})$ is a $q$-space;
\item $X$ is hemicompact.
\end{enumerate}
\end{corollary}

\bigskip

\bigskip

\section{On strong Choquet property}

\bigskip

In this section we prove that if $X$ is locally compact or a topological sum of $\sigma$-compact spaces, the space $B_1(X)$ equipped with the topology of uniform convergence on compacta, is a strong Choquet space.

\bigskip

\begin{lemma} \label{Alexandroff}
Let $X$ be a locally compact Tychonoff space. Let $C$ and $K$ be two compact sets in $X$ such that $C \subset$ Int$K$. Then there is a zero-set $D$ such that $C \subset D \subset$ Int$K$.
\end{lemma}
\begin{proof}
Let $X^*$ be the Alexandroff extension of $X$. Since $X^*$ is a compact Hausdorff space, it is normal. By Tietze extension theorem there is a continuous function $f: X^* \to \Bbb R$ such that $f(x) = 0$ for every $x \in C$ and $f(x) = 1$ for every $x \notin $ Int$K$. Thus $C \subset f^{-1}(\{0\}) \subset $ Int$K$. Since $\{x \in X^*: f(x) = 0\} = \{x \in X: f\upharpoonright_{X}(x) = 0\}$, we are done.
\end{proof}

\begin{lemma} \label{Alexandroff1}
Let $X$ be a locally compact Tychonoff space, $C$ and $K$ be two compact sets in $X$ such that $C \subset$ Int$K$ and $a \in \Bbb R$. Put $L = C \cup (X \setminus$ Int$K)$. If $f: L \to \Bbb R$ is a continuous function, such that $f(x) = a
$ for every $x \in X \setminus$ Int$K$, then there is a continuous extension $f^*: X \to \Bbb R$ of $f$.
\end{lemma}
\begin{proof}
Let $X^*$ be the Alexandroff extension of $X$. Put $L^* = L \cup \{\infty\}$ and define $g: L^* \to \Bbb R$ as $g(x) = f(x)$ for $x \in L$ and $g(\infty) = a$. Then $L^*$ is closed in $X^*$ and $g: L^* \to \Bbb R$ is continuous. Since $X^*$ is a compact Hausdorff space, it is normal. By Tietze extension theorem there is a continuous function $g^*: X^* \to \Bbb R$ such that $g^*(x) = g(x)$ for every $x \in L^*$. Then the function $f^* = g^*\upharpoonright_X$ is a continuous extension of $f$.
\end{proof}

\begin{lemma} \label{pointwise}
Let $X$ be a Tychonoff space and $(f_n)_{n \in \omega}$ be a sequence in $B_1(X)$, which pointwise converges to a function $f : X \to \Bbb R$.
If for each $\varepsilon >0$ there is a cover $(X_i)_{i \in \omega}$ of $X$ consisting of zero-sets such that $\text{diam}f_n(X_i) \le\varepsilon$, for all $i, n \in \omega$, then $f$ is Baire 1.
\end{lemma}
\begin{proof}
We will prove $(2)$ from Theorem \ref{HL16}. Let $\varepsilon > 0$. By the assumption there is a cover $(X_i)_{i \in \omega}$ of $X$ consisting of zero-sets such that $\text{diam}f_n(X_i) \le\varepsilon/3$, for all $i, n \in \omega$. We will show that $\text{diam}f(X_i) \le\varepsilon$ for every $i \in \omega$.
Let $i \in \omega$. Let $x, y \in X_i$. There is $n \in \omega$ such that $|f(x) - f_n(x)| < \epsilon/3$ and $|f(y) - f_n(y)| < \epsilon/3$. Thus we have
\bigskip

\centerline{$|f(x) - f(y)| \le |f(x) - f_n(x)| + |f_n(x) - f_n(y)| + |f_n(y) - f(y)| < \epsilon.$}

\bigskip
Then $\text{diam}f(X_i) \le\varepsilon$. By Theorem \ref{HL16} $f$ is Baire 1.
\end{proof}

\begin{lemma} \label{finite}
Let $X$ be a Tychonoff space and $\mathcal F \subset B_1(X)$ be a finite family. Then for each $\varepsilon >0$ there is a cover $(X_i)_{i \in \omega}$ of $X$ consisting of zero-sets such that $\text{diam}f(X_i) \le\varepsilon$, for all $i \in \omega$ and every $f \in \mathcal F$.
\end{lemma}
\begin{proof} $Zer_\sigma$-sets are preserved by finite
intersections.
\end{proof}

\bigskip

\begin{theorem} \label{local}
Let $X$ be a locally compact space. Then $(B_1(X),\tau_{UC})$ is a strong Choquet space.
\end{theorem}
\begin{proof}
We will define a winning strategy $s$ for the player $\alpha$ in $G^s(B_1(X))$.
Let $(f_0,U_0)$ be the first move of the player $\beta$. There is a compact set $K_0$ and positive $\epsilon_0$ such that $W(f_0,K_0,\epsilon_0) \subset U_0$. There are a compact set $K_1 \subset X$
and positive $\epsilon_1$ such that $K_0 \subset$ Int$K_1$, $\epsilon_1 < \epsilon_0/2$ and
$W(f_0,K_1,\epsilon_1) \subset W(f_0,K_0,\epsilon_0)$. Put $s((f_0,U_0)) = V_0 = W(f_0,K_1,\epsilon_1/2)$.
Let $(f_1,U_1)$ be the next move of $\beta$. Then $f_1 \in U_1 \subset W(f_0,K_1,\epsilon_1/2)$.
There are a compact set $K_2$ and positive $\epsilon_2$ such that $K_1 \subset$ Int$K_2$, $\epsilon_2 < \epsilon_1/2$ and $W(f_1,K_2,\epsilon_2) \subset U_1$. Put $s((f_0,U_0), V_0, (f_1,U_1)) = V_1 = W(f_1,K_2,\epsilon_2/2)$.
Suppose that for $n \in \omega$ we have a position

\bigskip

\centerline{$(f_0,U_0), V_0, ..., V_{n-1}, (f_n,U_n)$,}

\bigskip

and compact sets $K_i$, $\epsilon_i$ and $f_i \in B_1(X)$, $i \le n$ satisfying: $f_0 \in U_0$ and for $i \ge 1$
$f_i \in U_i \subset V_{i-1}$, $V_{i-1} = W(f_{i-1},K_i,\epsilon_i/2)$ and $K_{i-1} \subset$ Int$K_i$, $\epsilon_i < \epsilon_{i-1}/2$.
There are a compact set $K_{n+1}$ and positive $\epsilon_{n+1}$ such that
$K_n \subset$ Int$K_{n+1}$, $\epsilon_{n+1} < \epsilon_n/2$ and $W(f_n,K_{n+1},\epsilon_{n+1}) \subset U_n$. Put
$$s((f_0,U_0), V_0, ..., V_{n-1}, (f_n,U_n)) = V_n = W(f_n,K_{n+1},\epsilon_{n+1}/2).$$

We will prove that $s$ is the winning strategy for $\alpha$.
Consider an arbitrary run consistent with $\alpha$. Observe that $G = \bigcup_{n\in\omega} K_n = \bigcup_{n\in\omega}$ Int$K_n$ is open. For every $n \in \omega$ we have $K_n \subset $ Int$K_{n+1}$. By Lemma \ref{Alexandroff} there is a zero set $D_n$ such that $K_n \subset D_n \subset $ Int$K_{n+1}$. Thus the set $G = \bigcup_{n\in\omega} D_n$ is $Zer_\sigma$-set.

For every $x\in G$, the sequence $(f_n(x))_{n\in\omega}$ is Cauchy. Let $\epsilon > 0$. There is $n_0\in\omega$ such that $x\in K_{n_0}$ and $\epsilon_0/2^{n_0} < \epsilon$. For every $n\ge m\ge n_0$ we have that $f_n\in W(f_m,K_{m+1},\epsilon_{m+1}/2)$. Then
$|f_n(x) - f_m(x)|<\epsilon_{m+1}/2< \epsilon_0/2^{n_0 + 2} < \epsilon$. There is a function $f: G \to Y$ that is a pointwise limit of $(f_n\upharpoonright_{G})_{n\in\omega}$.

Now we prove that for every $\epsilon > 0$ there is a countable cover $(X_i)_{i \in \omega}$ of $G$ consisting of zero-sets such that $\text{diam}f_n(X_i) \le\varepsilon$, for all $i, n \in \omega$.

Let $\epsilon > 0$. There is $m \in \omega$ such that $\epsilon_0/2^{m} < \epsilon/3$. Fut $\mathcal F = \{f_i: i \le m\}$. By Lemma \ref{finite} there is a cover $(X_i^m)_{i \in \omega}$ of $X$ consisting of zero-sets such that $\text{diam}f(X_i^m) \le\varepsilon/3$, for all $i \in \omega$ and every $f \in \mathcal F$. Then the family $(X_i^m \cap D_m)_{i \in \omega}$ is a cover of $D_m$ consisiting of zero-sets such that $\text{diam}f(X_i^m \cap D_m) \le\varepsilon/3$, for all $i \in \omega$ and every $f \in \mathcal F$.
For every $n > m$ we have that $f_n\in W(f_m,K_{m+1},\epsilon_{m+1}/2)$. Then
$|f_n(x) - f_m(x)|<\epsilon_{m+1}/2< \epsilon_0/2^{m + 2} < \epsilon/3$ for every $x \in K_{m+1}$, thus $|f_n(x) - f_m(x)| < \epsilon/3$ also for every $x \in D_m$. Let $x, y \in X_i^m \cap D_m$, $i \in \omega$. Let $n > m$. Then $|f_n(x) - f_n(y)| \le |f_n(x) - f_m(x)| + |f_m(x) - f_m(y)| + |f_m(y) - f_n(y)| < \epsilon.$ Thus the family
$(X_i^m \cap D_m)_{i \in \omega}$ is a cover of $D_m$ consisiting of zero-sets such that $\text{diam}f_k(X_i^m \cap D_m) \le\varepsilon$, for all $i, k \in \omega$.

Now we will consider the situation on the set $D_k$ for $k > m$. We will use the same idea as above. Fut $\mathcal F = \{f_i: i \le k\}$. By Lemma \ref{finite} there is a cover $(X_i^k)_{i \in \omega}$ of $X$ consisting of zero-sets such that $\text{diam}f(X_i^k) \le\varepsilon/3$, for all $i \in \omega$ and every $f \in \mathcal F$. Then the family $(X_i^k \cap D_k)_{i \in \omega}$ is a cover of $D_k$ consisiting of zero-sets such that $\text{diam}f(X_i^k \cap D_k) \le\varepsilon/3$, for all $i \in \omega$ and every $f \in \mathcal F$.
For every $n > k$ we have that $f_n\in W(f_k,K_{k+1},\epsilon_{k+1}/2)$. Then
$|f_n(x) - f_k(x)|
<\epsilon_{k+1}/2 < \epsilon_0/2^{k + 2} < \epsilon_0/2^{m + 2} < \epsilon/3$ for every $x \in K_{k+1}$, thus $|f_n(x) - f_k(x)| < \epsilon/3$ also for every $x \in D_k$. The family
$(X_i^k \cap D_k)_{i \in \omega}$ is a cover of $D_k$ consisiting of zero-sets such that $\text{diam}f_l(X_i^k \cap D_k) \le\varepsilon$, for all $i, l \in \omega$.

Then the family
$(X_i^k \cap D_k)_{i, k \in \omega}$ is a cover of $G$ consisiting of zero-sets such that $\text{diam}f_l(X_i^k \cap D_k) \le\varepsilon$, for all $i, k, l \in \omega$. By Lemma \ref{pointwise},
 the function $f: G \to \Bbb R$ is Baire 1. Also

\bigskip

\centerline{ $f\in W(f_m\upharpoonright_G,K_{m+1},\epsilon_{m+1})$, for every $m\in\omega,$}
\bigskip

since for every $n > m$ we have that $f_n\in W(f_m,K_{m+1},\epsilon_{m+1}/2)$. It is sufficient to extend the function $f: G \to \Bbb R$ to a Baire 1 function $f^*: X \to \Bbb R$.

Let $(h_n)_{n \in \omega}$ be a sequence of continuous functions from $G$ to $\Bbb R$ such that $f(x) = $ lim$h_n(x)$ for every $x \in G$. We will define a pointwise convergent sequence of continuous functions from $X$ to $\Bbb R$, which will converge to $f$ on $G$.
 Let $n \in \omega$. Then $K_n \subset $ Int$K_{n+1}$. Put $L_n = K_n \cup (X \setminus $ Int$K_{n+1})$ and define $g_n: L_n \to \Bbb R$
as follows: $g_n(x) = h_n(x)$ if $x \in K_n$ and $g_n(x) = 1$ otherwise. By Lemma \ref{Alexandroff1}, there is a continuous extension $g_n^*: X \to \Bbb R$ of $g_n$. Define a function $f^*: X \to \Bbb R$ as follows: $f^*(x) = f(x)$, if $x \in G$ and $f(x) = 1$ if $x \in X \setminus G$. The sequence $(g_n^*)_{n \in \omega}$ is a sequence of continuous functions from $X$ to $\Bbb R$ pointwise convergent to $f^*$.
\end{proof}

\bigskip

\begin{lemma} \label{pointwise1}
Let $X$ be a topological space space and $(f_n)_{n \in \omega}$ be a sequence in $F_\sigma(X)$, which pointwise converges to a function $f : X \to \Bbb R$.
If for each $\varepsilon >0$ there is a cover $(X_i)_{i \in \omega}$ of $X$ consisting of closed sets such that $\text{diam}f_n(X_i) \le\varepsilon$, for all $i, n \in \omega$, then $f$ is also $F_\sigma$ measurable.
\end{lemma}
\begin{proof} We will use Theorem \ref{HL15}.
\end{proof}

\begin{theorem} \label{compact}
Let $X$ be a Tychonoff $\sigma$-compact space. Then $(B_1(X),\tau_{UC})$ is a strong Choquet space.
\end{theorem}
\begin{proof}
We will define a winning strategy $s$ for the player $\alpha$ in $G^s(B_1(X))$. Let $(C_n)_{n \in \omega}$ be a sequence of compact sets in $X$ such that $C_n \subset C_{n+1}$ for every $n \in \omega$ and $X = \bigcup_{n \in \omega} C_n$.
Let $(f_0,U_0)$ be the first move of the player $\beta$. There is a compact set $K_0$ and positive $\epsilon_0$ such that $W(f_0,K_0,\epsilon_0) \subset U_0$. There are a compact set $K_1 \subset X$
and positive $\epsilon_1$ such that $K_0 \cup C_0 \subset K_1$, $\epsilon_1 < \epsilon_0/2$ and
$W(f_0,K_1,\epsilon_1) \subset W(f_0,K_0,\epsilon_0)$. Put $s((f_0,U_0)) = V_0 = W(f_0,K_1,\epsilon_1/2)$.
Let $(f_1,U_1)$ be the next move of $\beta$. Then $f_1 \in U_1 \subset W(f_0,K_1,\epsilon_1/2)$.
There are a compact set $K_2$ and positive $\epsilon_2$ such that $K_1 \cup C_1 \subset K_2$, $\epsilon_2 < \epsilon_1/2$ and $W(f_1,K_2,\epsilon_2) \subset U_1$. Put $s((f_0,U_0), V_0, (f_1,U_1)) = V_1 = W(f_1,K_2,\epsilon_2/2)$.
Suppose that for $n \in \omega$ we have a position

\bigskip

\centerline{$(f_0,U_0), V_0, ..., V_{n-1}, (f_n,U_n)$,}

\bigskip

and compact sets $K_i$, $\epsilon_i$ and $f_i \in B_1(X)$, $i \le n$ satisfying: $f_0 \in U_0$ and for $i \ge 1$
$f_i \in U_i \subset V_{i-1}$, $V_{i-1} = W(f_{i-1},K_i,\epsilon_i/2)$ and $K_{i-1} \cup C_{i-1} \subset K_i$ and $\epsilon_i < \epsilon_{i-1}/2$.
There are a compact set $K_{n+1}$ and positive $\epsilon_{n+1}$ such
$K_n \cup C_n \subset K_{n+1}$, $\epsilon_{n+1} < \epsilon_n/2$ and $W(f_n,K_{n+1},\epsilon_{n+1}) \subset U_n$.
Put $s(f_0,U_0), V_0, ..., V_{n-1}, (f_n,U_n)) = V_n = W(f_n,K_{n+1},\epsilon_{n+1}/2)$.

We will prove that $s$ is the winning strategy for $\alpha$.
Consider an arbitrary run consistent with $\alpha$. Observe that $X = \bigcup_{n\in\omega} K_n = \bigcup_{n\in\omega}C_n$.

For every $x\in X$, the sequence $(f_n(x))_{n\in\omega}$ is Cauchy. Let $\epsilon > 0$. There is $n_0\in\omega$ such that $x\in K_{n_0}$ and $\epsilon_0/2^{n_0} < \epsilon$. For every $n\ge m\ge n_0$ we have that $f_n\in W(f_m,K_{m+1},\epsilon_{m+1}/2)$. Then
$|f_n(x) - f_m(x)| <\epsilon_{m+1}/2< \epsilon_0/2^{n_0 + 2} < \epsilon$. There is a function $f: X \to \Bbb R$ that is a pointwise limit of $(f_n)_{n\in\omega}$.

Using Theorem \ref{HL15} and a similar idea as in the proof of Theorem \ref{local} we prove that for every $\epsilon > 0$ there is a countable cover $(X_i)_{i \in \omega}$ of $X$ consisting of closed sets such that $\text{diam}f_n(X_i) \le\varepsilon$, for all $i, n \in \omega$. By Lemma \ref{pointwise1} and Theorem \ref{HL15}, the function $f$ is also $F_\sigma$ measurable. Since a $\sigma$-compact Tychonoff space is normal, by Proposition \ref{la}, the function $f$ is Baire 1. Also

\bigskip

\centerline{ $f\in W(f_m,K_{m+1},\epsilon_{m+1})$, for every $m\in\omega,$}
\bigskip

since for every $n > m$ we have that $f_n\in W(f_m,K_{m+1},\epsilon_{m+1}/2)$.
\end{proof}

\begin{corollary} \label{Baire}
Let $X$ be a countable Tychonoff space. Then $(B_1(X),\tau_{UC})$ is a strong Choquet space.
\end{corollary}

\begin{corollary}
Let $Q$ be the space of rationals equipped with the usual topology. Then $(B_1(Q),\tau_{UC})$ is a strong Choquet space.
\end{corollary}

The following Example shows that Corollary \ref{Baire} does not hold for $C(X)$.

\begin{example} Let $\omega$ be equipped with the discrete topology and $\beta \omega$ be the \v Cech-Stone compactification of $\omega$. Choose $q \in \beta \omega \setminus \omega$. Let $X = \omega \cup \{q\}$ and $X$ has the topology inherited from $\beta \omega$. Every compact set in $X$ is finite, thus the topology of uniform convergence on compacta on $C(X)$ coincides with the topology of pointwise convergence on $C(X)$. It was shown in \cite{LM} that $C(X)$ with the topology of pointwise convergence is Baire but not Choquet. Thus $(C(X),\tau_{UC})$ cannot be strong Choquet, since every strong Choquet space is Choquet \cite{Kech}. By Corollary \ref{hemicompact} the space $(B_1(X),\tau_{UC})$ is completely metrizable.
\end{example}

\bigskip

The condition of local compactness in Theorem 3.5 and the condition of $\sigma$-compactness in Theorem 3.7 are essential.

\begin{proposition} Let $\Bbb S$ be the Sorgenfrey line. Then $(B_1(\Bbb S),\tau_{UC})$ is of the first Baire category in itself.
\end{proposition}
\begin{proof} The Sorgenfrey line $\Bbb S$ is a hereditarily Baire, perfectly normal space in which every compact set is countable. Enumerate all intervals of the type $(a,b)$, where $a < b$ and $a, b$ are rationals as $\{I_n: n \in \omega\}$. For every $n \in \omega$ put

$$L_n = \{f \in B_1(\Bbb S): diam f(I_n) \le 1\}.$$

\bigskip

We claim that $B_1(\Bbb S) = \bigcup_{n \in \omega} L_n$. Let $f \in B_1(\Bbb S)$. Since $f$ is $F_\sigma$-measurable, by Remark \ref{lub} and Theorem \ref{HL15} there is  a cover $(X_i)_{i \in \omega}$ of $\Bbb S$ consisting of closed sets in $\Bbb S$ such that $\text{diam}f(X_i) \le 1$, for all $i \in \omega$. Since $\Bbb S$ is a Baire space, there is $i \in \omega$ such that the set $X_i$ has a nonempty interior in $\Bbb S$. There is $n \in \omega$ such that $I_n \subset X_i$, so $\text{diam}f(I_n) \le 1$. Thus $f \in L_n$.

It is easy to verify that $L_n$ is a closed set in $(B_1(\Bbb S),\tau_{UC})$ for every $n \in \omega$. We prove that $L_n$ is a nowhere dense set in $(B_1(\Bbb S),\tau_{UC})$ for every $n \in \omega$. Suppose there is $n \in \omega$ such that the set $L_n$ has a nonempty interior in $(B_1(\Bbb S),\tau_{UC})$. There is $f \in B_1(\Bbb S)$, a compact set $K$ in $\Bbb S$ and $\epsilon > 0$ such that $W(f,K,\epsilon) \subset L_n$. Since $K$ is countable,  $I_n \setminus K \ne \emptyset$. Let $a_n \in I_n \setminus K$. Put $M = f(a_n) + 2$. Define the function $g$ as follows: $g(x) = f(x)$ for every $x \in K$ and $g(x) = M$ otherwise. Since $\Bbb S \setminus K$ is $F_\sigma$-set, the function $g$ is $F_\sigma$-measurable. $\Bbb S$ is a normal space, thus by Proposition \ref{la}, $g$ is Baire 1. However $\text{diam}g(I_n) > 1$, so $g \in W(f,K,\epsilon) \setminus L_n$, a contradiction.

\end{proof}

\bigskip

\bigskip

\bigskip

\section{$(B_1(X),\tau_{UC})$ for a topological sum $X$}

\bigskip

\bigskip

\begin{lemma} \label{sum}
If $X= \bigoplus_{i\in I} X_i$, is a topological sum of topological spaces $X_i$, $i \in I$, then $(B_1(X),\tau_{UC})$ is homeomorphic to the product $\prod_{i \in I} (B_1(X_i),\tau_{UC})$.
\end{lemma}
\begin{proof} Define the mapping $\Psi: (B_1(X),\tau_{UC}) \to \prod_{i \in I} (B_1(X_i),\tau_{UC})$ as follows: $\Psi(f) = (f|X_i)_{i \in I}$ for $f \in B_1(X$. For every $i \in I$ the function $f|X_i \in B_1(X_i)$. It is easy to verify that $\Psi$ is a bijection. If $(f_i)_{i \in I} \in \prod_{i \in I} B_1(X_i)$,
then $\Psi^{-1}((f_i)_{i \in I})$ is the following function: $\Psi^{-1}((f_i)_{i \in I})(x) = f_i(x)$ if $x \in X_i$. Of course, $\Psi^{-1}((f_i)_{i \in I})$ is a Baire 1 function from $X$ to $\Bbb R$. Now we prove that $\Psi$ is homeomorphisms.
Suppose that the net $\{f_\sigma:\sigma \in \Sigma\}$ converges to $f$ in $(B_1(X),\tau_{UC})$. Let $O $ be an open neighbourhood of $\Psi(f)$ in $\prod_{i \in I} (B_1(X_i),\tau_{UC})$. There are a finite subset $J \subset I$, compact sets $K_j \subset X_j$, $j \in J$ and positive $\epsilon_j$, $j \in J$ such that

\bigskip

\centerline{$\prod_{j \in j} W(f|X_j,K_j,\epsilon_j) \times \prod_{i \in I \setminus J} B_1(X_i) \subset O.$}

\bigskip

Put $K = \cup_{j \in J} K_j$ and $\epsilon = min\{\epsilon_j: j \in J\}$. There is $\sigma_0 \in \Sigma$ such that for every $\sigma \ge \sigma_0$, $f_\sigma \in W(f,K,\epsilon)$. Thus for every $\sigma \ge \sigma_0$, $\Psi(f_\sigma) \in O$.

Suppose now that the net $\{F_\sigma:\sigma \in \Sigma\}$ converges to $F$ in $\prod_{i \in I} (B_1(X_i),\tau_{UC})$. Then $F = (f_i)_{i \in I}$ and $F_\sigma = (f_{\sigma,i})_{i \in I}$ for every $\sigma \in \Sigma$. Let $G$ be an open neighbourhood of $\Psi^{-1}((f_i)_{i \in I})$. There are a compact set $K \subset X$ and $\epsilon > 0$ such that

\bigskip
\centerline{$W(\Psi^{-1}((f_i)_{i \in I}),K,\epsilon) \subset G.$}
\bigskip
There is a finite subset $J $ of $I$ such that $K = \cup_{j \in J} K \cap X_j$. There is $\sigma_0 \in \Sigma$ such that

\bigskip
\centerline{$(f_{\sigma,i})_{i \in I} \in \prod_{j \in j} W(f_j,K \cap X_j,\epsilon) \times \prod_{i \in I \setminus J} B_1(X_i)$ for every $\sigma \ge \sigma_0.$}
\bigskip

Thus for every $\sigma \ge \sigma_0$ $\Psi^{-1}((f_{\sigma,i})_{i \in I}) \in W(\Psi^{-1}((f_i)_{i \in I}),K,\epsilon) \subset G.$
\end{proof}

\begin{theorem}
Let $X= \bigoplus_{i\in I} X_i$, be a topological sum of $\sigma$-compact spaces $X_i$, $i \in I$. Then $(B_1(X),\tau_{UC})$ is a strong Choquet space.
\end{theorem}
\begin{proof}
By Theorem \ref{compact} the space $(B_1(X_i),\tau_{UC})$ is a strong Choquet space for every $i \in I$. By Lemma \ref{sum} $(B_1(X),\tau_{UC})$ is homeomorphic to the product $\prod_{i \in I} (B_1(X_i),\tau_{UC})$. By \cite{Kech} any product of strong Choquet spaces is strong Choquet.
\end{proof}

\begin{proposition}
Let $X$ be a discrete space. Then $(B_1(X),\tau_{UC})$ is a strong Choquet space.
\end{proposition}
\begin{proof}
$B_1(X) = C(X) = \Bbb R^X$. Since $\Bbb R$ equipped with the usual Euclidean metric is a complete metric space, the product $\Bbb R^X$ is strong Choquet \cite{Kech}. Every compact set in $X$ is finite, thus the topology of uniform convergence on compacta on $B_1(X)$ coincides with the topology of pointwise convergence on $B_1(X)$.
\end{proof}

\begin{remark}
The result of the above Proposition for $X = \omega$ was mentioned in \cite{Os}.
\end{remark}

Recall that a topological space $X$ is called a $\lambda$-space \cite{Os} if every countable subset is $G_\delta$ in $X$. A subset $X$ of the real line $\Bbb R$ is called a $\lambda$-set if each countable subset $A \subset X$ is $G_\delta$ in $\Bbb R$.

\begin{definition}
A topological space $X$ is called a $\Lambda$-space if every countable union of compact sets is $G_\delta$ in $X$.
\end{definition}

\bigskip
\begin{theorem}
Let $X$ be a normal $\Lambda$-space. Then $(B_1(X),\tau_{UC})$ is strong Choquet.
\end{theorem}
\begin{proof}
We will define a winning strategy $s$ for the player $\alpha$ in $G^s(B_1(X))$. We will use a similar idea as in the proof of Theorem \ref{compact}.
Let $(f_0,U_0)$ be the first move of the player $\beta$. There is a compact set $K_0$ and positive $\epsilon_0$ such that $W(f_0,K_0,\epsilon_0) \subset U_0$. There are a compact set $K_1 \subset X$
and positive $\epsilon_1$ such that $K_0 \subset K_1$, $\epsilon_1 < \epsilon_0/2$ and
$W(f_0,K_1,\epsilon_1) \subset W(f_0,K_0,\epsilon_0)$. Put $s((f_0,U_0)) = V_0 = W(f_0,K_1,\epsilon_1/2)$.
Let $(f_1,U_1)$ be the next move of $\beta$. Then $f_1 \in U_1 \subset W(f_0,K_1,\epsilon_1/2)$.
There are a compact set $K_2$ and positive $\epsilon_2$ such that $K_1 \subset K_2$, $\epsilon_2 < \epsilon_1/2$ and $W(f_1,K_2,\epsilon_2) \subset U_1$. Put $s((f_0,U_0), V_0, (f_1,U_1)) = V_1 = W(f_1,K_2,\epsilon_2/2)$.
Suppose that for $n \in \omega$ we have a position

\bigskip

\centerline{$(f_0,U_0), V_0, ..., V_{n-1}, (f_n,U_n)$,}

\bigskip

and compact sets $K_i$, $\epsilon_i$ and $f_i \in B_1(X)$, $i \le n$ satisfying: $f_0 \in U_0$ and for $i \ge 1$
$f_i \in U_i \subset V_{i-1}$, $V_{i-1} = W(f_{i-1},K_i,\epsilon_i/2)$ and $K_{i-1} \subset K_i$ and $\epsilon_i < \epsilon_{i-1}/2$.
There are a compact set $K_{n+1}$ and positive $\epsilon_{n+1}$ such
$K_n \subset K_{n+1}$, $\epsilon_{n+1} < \epsilon_n/2$ and $W(f_n,K_{n+1},\epsilon_{n+1}) \subset U_n$.
Put $s(f_0,U_0), V_0, ..., V_{n-1}, (f_n,U_n)) = V_n = W(f_n,K_{n+1},\epsilon_{n+1}/2)$.
We will prove that $s$ is the winning strategy for $\alpha$.
Consider an arbitrary run consistent with $\alpha$. Put $L = \bigcup_{n\in\omega}K_n$.

For every $x\in L$, the sequence $(f_n(x))_{n\in\omega}$ is Cauchy. Let $\epsilon > 0$. There is $n_0\in\omega$ such that $x\in K_{n_0}$ and $\epsilon_0/2^{n_0} < \epsilon$. For every $n\ge m\ge n_0$ we have that $f_n\in W(f_m,K_{m+1},\epsilon_{m+1}/2)$. Then
$|f_n(x) - f_m(x)| <\epsilon_{m+1}/2< \epsilon_0/2^{n_0 + 2} < \epsilon$. There is a function $f: L \to \Bbb R$ that is a pointwise limit of $(f_n\upharpoonright_{L})_{n\in\omega}$.

Using Theorem \ref{HL15} and a similar idea as in the proof of Theorem \ref{local} we prove that for every $\epsilon > 0$ there is a countable cover $(X_i)_{i \in \omega}$ of $L$ consisting of closed sets such that $\text{diam}f_n(X_i) \le\varepsilon$, for all $i, n \in \omega$. By Lemma \ref{pointwise1} and Theorem \ref{HL15}, the function $f$ is also $F_\sigma$ measurable on the set $L$. Define the function $f^*: X \to \Bbb R$ as follows: $f^*(x) = f(x)$ if $x \in L$ and $f^*(x) = 1$ otherwise. Since the set $L$ is $G_\delta$, the function $f^*$ is $F_\sigma$ measurable. Since the space $X$ is normal, by Proposition \ref{la}, the function $f^*$ is Baire 1. Also

\bigskip

\centerline{ $f^*\in W(f_m,K_{m+1},\epsilon_{m+1})$, for every $m\in\omega,$}
\bigskip

since for every $n > m$ we have that $f_n\in W(f_m,K_{m+1},\epsilon_{m+1}/2)$.
\end{proof}

\bigskip

Recall that a topological space $X$ is called a $Q$-space \cite{BMZ} if every subset of $X$ is $F_\sigma$ in $X$. Of course, every $Q$-space is a $\Lambda$-space.

\bigskip

Under Martin’s Axiom and the negation of the Continuum Hypothesis, there exists an uncountable set $X \subset \Bbb R$ such that each subset of $X$ is $F_\sigma$ in $X$; see \cite{Bu}, Theorem 9.6.

\bigskip
\begin{corollary}
Let $X$ be a normal $Q$-space. Then $(B_1(X),\tau_{UC})$ is strong Choquet.
\end{corollary}

\bigskip

\bigskip

Acknowledgements. This work was supported by the Slovak Research and Development Agency under the Contract no. APVV-20-0045 and by the grant VEGA 2/0009/26.

\end{document}